\documentclass[11pt]{amsart}
\usepackage{amsmath,amssymb,amsthm,amsfonts,verbatim}
\usepackage[abbrev,alphabetic]{amsrefs}

\usepackage{enumerate}
\usepackage{color}
\usepackage{graphicx}

\usepackage{tikz-cd}
\usepackage[utf8]{inputenc}

\usepackage[pdfborder={0 0 0 [0 0 ]},bookmarksdepth=5]{hyperref}
\usepackage[capitalize]{cleveref}

\hypersetup{
    colorlinks=true,
    linkcolor=black,
    citecolor=black,
    filecolor=black,
    urlcolor=black,
}

\theoremstyle{plain}
\newtheorem{theorem}{Theorem}[section]

\newtheorem{proposition}[theorem]{Proposition}

\newtheorem{lemma}[theorem]{Lemma}

\newtheorem{corollary}[theorem]{Corollary}

\theoremstyle{definition}

\newcommand{\nc}{\newcommand}
\nc{\dmo}{\DeclareMathOperator}

\nc{\QQ}{\mathbb{Q}}
\nc{\RR}{\mathbb{R}}
\nc{\RP}{\mathbb{RP}^1}
\nc{\ZZ}{\mathbb{Z}}
\nc{\CC}{\mathbb{C}}
\nc{\iso}{\cong}
\dmo{\Mod}{Mod}
\dmo{\Ig}{\mathcal{I}_g}
\dmo{\Span}{span}
\dmo{\Diff}{Diff}
\dmo{\Homeo}{Homeo}
\dmo{\dist}{dist}
\dmo\BDiff{BDiff}
\dmo\SO{SO}
\dmo\slide{sl}
\dmo\im{im}
\dmo\rk{rk}
\dmo\id{id}
\dmo\Id{Id}
\dmo\Fix{Fix}
\dmo\Stab{Stab}
\dmo\Mcg{Mcg}
\dmo{\Hg}{\mathcal{H}_g}
\dmo{\Tg}{\mathcal{T}_g}

\renewcommand{\epsilon}{\varepsilon}
\nc{\coloneq}{\mathrel{\mathop:}\mkern-1.2mu=}
\nc{\margin}[1]{\marginpar{\scriptsize #1}}
\nc{\para}[1]{\bigskip\noindent\textbf{#1}}

\begin{document}
\title{Handlebody Bundles and Polytopes}
\author{Sebastian Hensel and Dawid Kielak}
\date{\today}
\begin{abstract}
  We construct examples of fibered three-manifolds with first Betti number at least $2$ and with fibered faces
  all of whose monodromies extend to a handlebody.
\end{abstract}
\maketitle

\section{Introduction}
\label{sec:intro}

Suppose that $M$ is an orientable three-manifold which fibers over the circle with
fiber a closed connected surface; let $\omega \colon \pi_1(M) \to \ZZ$ denote the induced homomorphism (we will say that $\omega$ is a \emph{fibered class} in the first cohomology of $M$). 
Thurston
\cite{Thurston1986} developed a theory which describes all possible
ways in which $M$ can fiber. Namely, he constructed a convex polytope
$P_M$ in $H^1(M;\mathbb{R})$ such that the fibered classes of $M$ are exactly the integral classes in cones over
certain ``fibered'' faces of $P_M$.

In particular, all the integral classes $\omega'$ in the cone $C_F$ containing the
class $\omega$ 
are also fibered. For
each such $\omega'$, there is an associated \emph{monodromy homeomorphism}
(determined up to isotopy). As all of these monodromies describe the same $3$--manifold,
these elements of mapping class groups \emph{of different surfaces} should share
interesting properties -- but is in general extremely hard to explicitly relate them.

In this article we present a construction of three-manifolds in which
all of these monodromies share the property that they extend to
handlebodies. Namely, we show:
\begin{theorem}\label{thm:intro1}
There exists infinitely many pairwise non-diffeomorphic, closed three-manifolds $M$ with $b_1(M)\geqslant 2$ and with the following property: the Thurston polytope $P_M$ of $M$ contains a fibered face $F$ such that every integral class in the cone $C_F$ over $F$ is fibered, and its associated monodromy extends from the closed surface on which it is defined to a handlebody.
\end{theorem}

The proof of this theorem relies on a connection of handlebody bundles
to free-by-cyclic groups; the latter have recently been studied
in analogy to fibered three-manifolds, see e.g.~\cites{Dowdalletal2015,Dowdalletal2017a,Dowdalletal2017,FunkeKielak2018,Kielak2018a,Mutanguha2019}.
 Formally, \cref{thm:intro1} follows from \cref{thm:all-handlebody,free-by-cyclic} below.

To elucidate the connection, we need the following definition.  We say
that a class $\omega$ is \emph{compatible with a handlebody bundle} if
$\omega$ is induced by $M$ fibering over the circle with monodromy a
mapping class $\varphi$ of some closed surface $S_g$, such that
$\varphi$ extends to a mapping class of a handlebody $V_g$. We say
that $\omega$ is \emph{fully compatible with a handlebody bundle} if
in addition the inclusion map $M \hookrightarrow W$ induces an
isomorphism $H_1(M;\QQ) \cong H_1(W;\QQ)$
where $W$ denotes the fibered four-manifold whose monodromy is the extension of $\varphi$ to $V_g$.

The fundamental group of $W$ is a
free-by-cyclic group $\Gamma = \pi_1(W)$ fitting into the following commutative diagram:
\begin{center}
  \begin{tikzcd}
    1 \arrow[r] & \pi_1(S_g) \arrow[r]\arrow[d, "\iota"] & \pi_1(M) \arrow[r,"\omega"]\arrow[d, "\hat{\iota}"] & \ZZ \arrow[r]\arrow[d, "="] & 1 \\
    1 \arrow[r] & \pi_1(V_g) \arrow[r] & \Gamma \arrow[r,"\omega_\Gamma"] & \ZZ \arrow[r] & 1 \\
  \end{tikzcd}
\end{center}
where $\omega_\Gamma$ is induced by $\omega$, and where $\iota$ and $\hat \iota$ are epimorphisms induced by the embeddings $S_g \hookrightarrow V_g$ and $M \hookrightarrow W$.

In recent work \cite{Kielak2018a}, the second author
constructed a convex polytope $P_\Gamma$ which serves as an analogue of the
Thurston polytope $P_M$, classifying fiberings of $\Gamma$, i.e.  maps
$\Gamma\to\mathbb{Z}$ with finitely generated kernel.

With this terminology, our main result is:
\begin{theorem}\label{thm:intro2}
\label{thm:all-handlebody}
  Let $M$ be a closed three-manifold, and let $\omega \in H^1(M;\ZZ)$
  be fully compatible with a handlebody bundle.
  If $F$ denotes the fibered face whose cone $C_F$ contains $\omega$, then every integral class $\omega' \in C_F$ is fully compatible with some handlebody bundle.
%
\end{theorem}
The condition that the inclusion $M \hookrightarrow W$ should induce an isomorphism
on $H_1$ (required by the definition of full compatibility) is easy to check, and grants us the flexibility to prove the
following application.
\begin{theorem}
\label{free-by-cyclic}
  Suppose that $\Gamma$ is a
  free-by-cyclic group. Then there are infinitely many pairwise non-diffeomorphic,
  hyperbolic three-manifolds admitting fibered classes fully compatible with handlebody bundles with fundamental group $\Gamma$.
\end{theorem}
Finally, we want to emphasise that the theorem yields a new connection between
mapping classes of surfaces and (outer) automorphisms of free groups; to \emph{every} automorphism we can associate (infinitely many different) mapping classes, which should inherit properties from the free group automorphism.
To our knowledge, this intriguing connection has not yet been explored. 

\subsection*{Acknowledgements} The authors would like to thank the organizers of the ``Moduli Spaces" conference on Ventotene in 2017, where some of this work was conducted.

The second author was supported by the grant KI 1853/3-1 within the Priority Programme 2026 \href{https://www.spp2026.de/}{``Geometry at Infinity''} of the German Science Foundation (DFG).

\section{The Thurston polytope for three-manifolds and free-by-cyclic groups}

Throughout, we will use the notation established in the introduction: $M$ is a closed, connected and oriented three-manifold which fibres over the circle with fiber $S_g$, associated class $\omega \in H^1(M;\ZZ)$ and monodromy $\varphi$.

A group will be called \emph{free-by-cyclic} if it is an extension of a finitely generated free group by $\ZZ$. This is the case for the fundamental group of every handlebody bundle $W$ with which $\omega$ is compatible.

Thurston~\cite{Thurston1986}*{Theorem 5} proved that all the different ways in which a given three-manifold can fiber over the circle are encoded by a polytope, in a way which we will make precise in a moment. The second author gave a new proof of Thurston's theorem in \cite{Kielak2018a}*{Theorem 5.34}, and then extended the result to cover also free-by-cyclic groups \cite{Kielak2018a}*{Theorem 5.29} -- in this latter setting, `fibering over the circle' is interpreted to mean the existence of an epimorphism to $\ZZ$ with a finitely generated kernel.

Note that, thanks to a result of Stallings~\cite{Stallings1962}, an integral cohomology $1$-class $\omega \colon \pi_1(M) \to \ZZ$ of an irreducible three-manifold $M$ is fibered if and only if $\ker \omega$ is finitely generated. (Recall that $\omega$ being fibered means that it is induced by $M$ fibering over the circle.) Moreover, one can remove the assumption of $M$ being irreducible thanks to Perelman's solution of the Poincar\'e conjecture. Hence, the group-theoretic notion of fibering used above coincides with the topological one for three-manifold groups.

It is important to note that if the kernel of $\omega \colon G \to \ZZ$ is finitely generated (that is, if $\omega$ is \emph{fibered}), then $\ker \omega$ is in fact a surface group or a free group if $G$ is a three-manifold group, and a free group if $G$ is a free-by-cyclic group (the latter by \cite{Geogheganetal2001}). In either case, the kernel has a well-defined Euler characteristic, denoted by $\chi(\ker \omega)$.

Before proceeding, let us state some definitions: a \emph{polytope} in a finite-dimensional $\RR$-vector space $V$ denotes the intersection of finitely many halfspaces, and therefore a polytope $P$ must be convex, but need not be compact. Unless explicitly stated to the contrary, we will also require polytopes to be \emph{symmetric}, that is preserved by the map $v \mapsto -v$. Given a face $F \neq \{0\}$ of a polytope, which for us will always mean an open face, we define $C_F$ to be the cone over that face; explicitly, we set
\[                                                                                                                                                                                                                                                                                                                                                                                                                                                                                                                                                                                                                                                                                                                                                                                                                                                                          C_F = \{ \lambda v \mid v \in F, \lambda \in (0,\infty) \}.                                                                                                                                                                                                                                                      \]
When $P = F =\{0\}$ we define $C_F = V$.

\begin{theorem}[{\cites{Thurston1986,Kielak2018a}}]
\label{thm:polytope}
Suppose that $G$ is a three-manifold group or a free-by-cyclic group. 
There exists a
    polytope $P$ in $H^1(G;\RR)$ such that for every epimorphism $\omega\colon G \to \ZZ$ with a finitely generated kernel there exists a face $F$ (the associated \emph{fibered face}) of $P$ with $\omega \in C_F$  such that
    \begin{enumerate}
    \item $F$ is top-dimensional, or equivalently, $C_F$ is open, and
     \item every primitive integral class $\omega' \in H_1(G;\ZZ)$ lying in $C_F$
    has a finitely generated kernel, and
    \item the map $\omega' \mapsto \chi(\ker \omega')$ defined on the primitive integral classes in $C_F$ extends to a linear functional defined on the whole of $C_F$.
    \end{enumerate}
\end{theorem}
\begin{proof}
We begin by examining the pathological case of $G = \mathbb Z$. Note that $\mathbb Z$ is at the same time the fundamental group of a closed three-manifold, and  a free-by-cyclic group.

In this case we take $P = \{0\}$. The unique face of $P$ is maximal dimensional, and its cone $V = \mathbb R$ is open. There are exactly two primitive integral classes, and their kernels are the trivial group, which certainly is finitely generated. The functional $v \mapsto |v|$ extends the map $\omega' \mapsto \chi(\ker \omega') = 1$. In what follows we will assume that $G \not \cong \mathbb Z$.

\smallskip

Let us start from the more classical case, in which $G$ is a three-manifold group which is not $\mathbb S^2 \times \mathbb S^1$.
The polytope $P$
above is what is denoted by $B_{x}$ in \cite{Thurston1986}, and is the unit ball of the Thurston norm $x \colon H^1(G;\RR) \to [0,\infty)$. 
The Thurston norm $x(\omega')$ of a primitive cohomology class $\omega' \in H^1(G;\ZZ)$ with finitely generated kernel is equal to $- \chi (\ker \omega')$ by definition.
For convenience, we define $N = \frac 1 2 x \colon H^1(G;\RR) \to [0,\infty)$.

It is immediate that $x$, and hence $N$, are linear on $C_F$.

The facts that $C_F$ is open and that every primitive integral class therein is fibered follow from \cite{Thurston1986}*{Theorems 3 and 5}.

\smallskip
Now suppose that $G$ is a free-by-cyclic group, where the free kernel is not trivial.
The starting point is the $L^2$-torsion polytope $P_{L^2} \subseteq H_1(G;\RR)$ appearing in \cite{Kielak2018a}*{Theorem 5.29}, and introduced first by Friedl--L\"uck~\cite{FriedlLueck2017}. Note that $P_{L^2}$ is in general not symmetric. The polytope $P_{L^2}$ induces a \emph{thickness function} $T \colon H^1(G;\RR) \to [0,\infty)$ by setting
\[
 T(\omega') = \max_{p,q \in P_{L^2}} \left\vert \omega'(p) - \omega'(q) \right\vert
\]
In fact, $T$ is a semi-norm by \cite{FunkeKielak2018}*{Corollary 3.5}, and if $\ker \omega'$ is finitely generated and $\omega'$ is primitive, then
\[
 T(\omega') = - \chi(\ker \omega')
\]
 by the proof of \cite{FunkeKielak2018}*{Theorem 4.4} (see also \cite{HennekeKielak2018}*{Theorem 6.2 and Remark 6.5}).

The desired polytope $P$ is defined to be the unit ball of the semi-norm $T$. This immediately implies that $T$ is linear on the cones of the faces of $P$. Note that, in general $P$ is not the same as $P_{L^2}$. In particular, $P$ is symmetric, whereas $P_{L^2}$ does not have to be.

Since $\ker \omega$ is finitely generated, we have $\omega$ and $-\omega$ lying in the (first) Bieri--Neumann--Strebel invariant $\Sigma^1(G)$ by \cite{Bierietal1987}*{Theorem B1}, and therefore \cite{Kielak2018a}*{Theorem 5.29} tells us that there are unique points $p$ and $q \in P_{L^2}$ such that $\omega$ restricted to $P_{L^2}$ attains its minimum at $p$ and maximum at $q$. But this is an open condition, and therefore $T$ is linear on a neighbourhood $U$ of $\omega$. This implies that the cone $C_F$ containing $\omega$ contains $U$, where $F$ is a face of $P$. Hence $C_F$ has non-empty interior. This implies that $F$ is maximal dimensional, and hence $C_F$ is open.

The cone $C_F$ consists of precisely these cohomology classes which, when restricted to $P_{L^2}$, attain their minimum precisely at $p$ and their maximum precisely at $q$. Therefore, every integral class in $C_F$ is fibered by \cite{Kielak2018a}*{Theorem 5.29}.
\end{proof}

\section{All Fiberings are Handlebody}

In this section we assume in addition to the assumptions of the last section that $\omega$ is fully compatible with a handlebody bundle $W$ which fibers with fiber a handlebody $V_g$.
We also let $F$ denote the fibered face of $P_M$, the Thurston polytope, whose cone $C_F$ contains $\omega$.
We set $\Gamma = \pi_1(W)$ as before.

As indicated in the introduction, we have the following diagram with exact rows:
\begin{center}
  \begin{tikzcd}
    1 \arrow[r] & \pi_1(S_g) \arrow[r]\arrow[d, "\iota"] & \pi_1(M) \arrow[r,"\omega"]\arrow[d, "\hat{\iota}"] & \ZZ \arrow[r]\arrow[d, "="] & 1 \\
    1 \arrow[r] & \pi_1(V_g) \arrow[r] & \pi_1(W) \arrow[r,"\omega_\Gamma"] & \ZZ \arrow[r] & 1 \\
  \end{tikzcd}
\end{center}
Here $\iota, \hat{\iota}$ are the maps induced by the inclusions of the boundary. Note that since $\iota$ is
surjective, so is $\hat{\iota}$.

Recall that we are also assuming that
the epimorphism $\hat{\iota}$ induces an isomorphism
\[ \hat{\iota}_*:H_1(M; \QQ) \to H_1(W; \QQ). \]

Let $F_k$ denote the free group of rank $k$.
We need the following ingredient:
\begin{proposition}[Co-rank theorem for surface groups]\label{prop:corank}
  If $f\colon \pi_1(S_g) \to F_k$ is a surjective map, then $k \leqslant g$. In the case of equality,
  the map $f$ is induced by the identification of $S_g$ with the boundary of a genus $g$ handlebody $V_g$.
  Furthermore, if in that case $\psi$ is any mapping class of $S_g$ which preserves $\ker(f)$, then
  $\psi$ has an extension to $V_g$.
\end{proposition}
\begin{proof}
  The fact that $k \leqslant g$ is \cite{corank}*{Lemma~2.1}, while the fact
  on the identification with a handlebody is
  \cite{corank}*{Lemma~2.2}. The fact that any mapping class of
  $\partial V_g$ which preserves $\ker(\pi_1(\partial V_g)\to \pi_1(V_g))$
  extends to the handlebody $V$ is standard, see e.g. \cite{survey}*{Corollary~5.11}.
\end{proof}

We are now ready to prove the main theorem.

\begin{proof}[{Proof of Theorem~\ref{thm:all-handlebody}}]
  Let $\omega' \colon \pi_1(M) \to \ZZ$ be an epimorphism lying in the cone $C_F$. Since
  $\pi_1(M) \to \pi_1(W)$ is surjective, and induces an isomorphism $H_1(M;\QQ) \cong H_1(W;\QQ)$,
%
 there is an epimorphism $\omega'_\Gamma$ that makes the right square in the following diagram commute:
\begin{center}
  \begin{tikzcd}
    1 \arrow[r] & \pi_1(S_h) \arrow[r]\arrow[d, "f"] & \pi_1(M) \arrow[r,"\omega'"]\arrow[d, "\hat{\iota}"] & \ZZ \arrow[r]\arrow[d, "="] & 1 \\
    1 \arrow[r] & \ker(\omega_\Gamma') \arrow[r] & \Gamma \arrow[r,"\omega'_\Gamma"] & \ZZ \arrow[r] & 1 \\
  \end{tikzcd}
\end{center}
By a simple diagram chase, a homomorphism $f$ which makes the left square commute exists, and is surjective.
Therefore, $H = \ker(\omega_W')$ is finitely generated. But $\Gamma = \pi_1(W)$ is a free-by-cyclic group, and hence $H$ is a free group by \cite{Geogheganetal2001}.

We now claim that the rank of $H$ is at least $h$. Suppose first that we have shown the claim. Now the co-rank theorem for surface groups (Proposition~\ref{prop:corank}) tell us that the rank is exactly $h$.
Let $x \in \ker f$, and let $z \in \pi_1(M)$ denote some preimage under $\omega'$ of a generator of $\ZZ$. We have
\[
 f(z^{-1} x z) = \hat\iota(z^{-1} x z) = \hat\iota(z^{-1} )\hat\iota( x )\hat\iota( z) = \hat\iota(z^{-1} )f( x )\hat\iota( z) = 1
\]
and so $\ker f$ is preserved by the monodromy induced by $\omega'$ (whose action coincides with  conjugation by $z$).
The second part of the co-rank theorem now gives us a homeomorphism of the
corresponding handlebody $V_h$ with boundary $S_h$
extending the monodromy induced by $\omega'$.

We are left with proving the claim. For a contradiction, suppose that
the rank $g$ of $H$ is strictly smaller than $h$. Write $v = \omega_\Gamma
- \omega_\Gamma'$; we then have $\hat{\iota}^*v = \omega- \omega'$.
Observe that by Theorem~\ref{thm:polytope}, there are
nondegenerate linear functionals $N$ (half of the Thurston norm) and $T$ (the
thickness function), such that
\[ g-h = \frac 1 2 \left( \chi(S_h) - \chi(S_g) \right) = N(\hat{\iota}^*(v))\]
and
\[\quad \rk\big(\ker(\omega_\Gamma)\big) - \rk\big(\ker(\omega_\Gamma')\big) = \chi\big(\ker(\omega_\Gamma')\big) - \chi\big(\ker(\omega_\Gamma)\big) = T(v) \]
Since $g = \rk(\ker(\omega_\Gamma))$, and $\rk(\ker(\omega_\Gamma')) < h$,
this implies
\[ N(\hat{\iota}^*(v)) < T(v). \]
Consider $\omega_\Gamma'' = \omega_\Gamma + q v$ and $\omega'' = \omega + q \hat \iota^* v$ for a small rational number $q$.
Since $q$ is rational, the cohomology class $\omega''$ is also rational, in the sense that $\omega'' \in H^1(M;\QQ)$. There exists a unique positive integer $k$ such that $k \omega''$ is integral and primitive. Also, since $q$ is small, $k \omega''$ lies in the cone of the same fibered face as $\omega$, and hence is a fibered character. Arguing as before using \cref{prop:corank} and \cite{Geogheganetal2001}, we thus have
\[
T(\omega''_\Gamma) =  \frac {-1} k \chi(\ker (k \omega''_\Gamma)) \leqslant \frac {-1} {2k} \chi(\ker (k \omega'')) = N(\omega'')
\]
We also have
\[
 N(\omega'') = N(\omega) + q N(\hat \iota^* v) < T(\omega_\Gamma) + q T(v) = T(\omega''_\Gamma)
\]
and so
\[
 T(\omega''_\Gamma) \leqslant N(\omega'')  < T(\omega''_\Gamma),
\]
which is a contradiction. We have therefore proven the claim.
\end{proof}

\section{Existence of fully compatible fibered classes}
\label{sec:crit}

In this section we show how to construct bundles $M$ that will satisfy the assumption
that there exists an isomorphism $\hat{\iota}_* \colon H_1(M;\QQ) \cong H_1(W;\QQ)$ where $\pi_1(W)$ is a given  free-by-cyclic group. More precisely, we will show the following,
which is a rephrasing of \cref{free-by-cyclic}.
\begin{theorem}\label{thm:manybundles}
  Given any free group automorphism $f:F_g \to F_g$, there are mapping classes
  $\varphi_i$ of the handlebody $V_g$ such that
  \begin{enumerate}[i)]
  \item $\varphi_i$ induces the automorphism $f$ on $\pi_1(V_g) = F_g$
    for all $i$.
  \item The (four-manifolds) $W_i$ obtained as the mapping tori of the mapping classes
    $\varphi_i$ satisfy $\hat{\iota}_* \colon H_1(\partial W_i;\QQ) \cong H_1(W_i;\QQ)$ for all $i$.
  \item The (three-manifolds) $M_i = \partial W_i$ are hyperbolic for all $i$ and are pairwise non-diffeomorphic.
  \end{enumerate}
\end{theorem}
Before we can give the proof, we need some basic notation.
Recall that if $S_g$ is a closed surface of
genus $g$, the algebraic intersection number defines a symplectic
pairing
\[  H_1(S_g; \ZZ) \times H_1(S_g; \ZZ) \to \ZZ \]
which extends in the obvious way to
\[ \sigma\colon H_1(S_g; \QQ) \times H_1(S_g; \QQ) \to \QQ \]
on the first homology group. Suppose now that $S_g$ has been identified with the boundary
$\partial V_g$ of a handlebody. Then, the inclusion of the boundary defines a map
\[ \iota_*\colon H_1(S_g;\QQ) \to H_1(V_g;\QQ) \]
whose kernel we denote by $L$. Explicitly, let $\alpha_1, \ldots, \alpha_g$ be disjoint
curves bounding disks which cut the handlebody $V_g$ into a ball. Choose curves $\beta_i$
with the property that $\sigma(\alpha_i, \beta_j)$ is $0$ if $i\neq j$ and $1$ otherwise. Then
the homology classes $a_i, b_j$ defined by the curves $\alpha_i, \beta_j$, respectively, are a basis for
$H_1(S_g;\QQ)$. We then have that
\[ L = \ker(\iota_*) = \langle a_1, \ldots, a_g \rangle. \]
Furthermore, if we define
\[ D = \langle b_1, \ldots, b_g \rangle, \]
then the restriction $\iota_*\colon D \to H_1(V_g;\QQ)$ is an isomorphism. Furthermore, $\sigma$
vanishes identically on $L$ and $D$. In other words, we have
\[ H_1(S_g;\QQ) = L \oplus D, \] and both $L, D$ are Lagrangian
subspaces. With respect to this decomposition, $\sigma$ corresponds to
the matrix
\[ J =
\begin{pmatrix}
  0 & \Id \\ -\Id & 0
\end{pmatrix}. \]
Denote by $\mathcal{H}_g < \mathrm{Mcg}(S_g)$ the \emph{handlebody group}, i.e. the subgroup of those
mapping classes of $S_g$ which extend to $V_g$.
If $\phi$ is an element of the handlebody group, then $\phi_*(L) = L$. This gives the following
obstruction for how the handlebody group acts on homology.
\begin{lemma}[{e.g. \cite{birman-matrix}*{Lemma~2.2}}]\label{lem:hanaction}
  For a symplectic basis as above, every handlebody group element $\phi$ acts
  on $H_1(S_g;\mathbb{Q})$ as a matrix of the form
  \[ \phi_* = \begin{pmatrix}
    A & B \\ 0 & (A^t)^{-1}
  \end{pmatrix}, \]
  where $A$ is invertible, $B$ satisfies $B^t(A^t)^{-1} = A^{-1}B$, and the entries are all integers.
  Conversely, any such matrix is realised as the action of a suitable handlebody group element $\phi$.

	Furthermore, $(A^t)^{-1}$ is the matrix representing the action of $\phi$ on $H_1(V_g, \QQ)$ with respect to the basis $b_1, \ldots, b_g$.
\end{lemma}
We also need the following variant, which is likely well-known to experts.
\begin{lemma}\label{lem:kernel-realise}
  For a basis of $H_1(S_g;\mathbb{Q})$ as above, every integral symplectic matrix of the form
  \[
  \begin{pmatrix}
    \Id & B \\
    0 & \Id
  \end{pmatrix}
  \]
  can be realised as the homology action of a handlebody mapping class which
  acts trivially on the fundamental group of the handlebody.
\end{lemma}
\begin{proof}
  The condition that the matrix is symplectic implies that $B$ has to
  be symmetric. First, let $i$ be given. The twist about $\alpha_i$
  acts as the matrix
  \[
  \begin{pmatrix}
    \Id & E_i \\
    0 & \Id
  \end{pmatrix}
  \]
  where $E_i$ is the matrix which is zero, except a single diagonal
  entry $1$ in column $i$.

  Next, let $i \neq j$ be given. Let $\delta$ be a diskbounding curve
  which intersects each of $\beta_i, \beta_j$ in a single point, and defines
  the homology class $a_i + a_j$. The twist about $\delta$ acts as
  \[
  \begin{pmatrix}
    \Id & E_i+E_j+E_{ij}+E_{ji} \\
    0 & \Id
  \end{pmatrix}
  \]
  where $E_{ij}$ is the elementary matrix with entry $1$ in row $i$,
  column $j$. Since Dehn twists about diskbounding curves extend to
  the handlebody, and their extensions act trivially on the
  fundamental group of the handlebody, the lemma is proved for
  matrices of the form $B=E_i$ and $B=E_{ij}+E_{ji}$. Since these
  (additively) generate the group of integral symmetric matrices, the lemma
  follows.
\end{proof}
To certify that $\hat\iota_* \colon H_1(M_i;\QQ) \cong H_1(W_i;\QQ)$ in the proof
of Theorem~\ref{thm:manybundles}, we will use the following criterion.
\begin{lemma}\label{lem:assumption-1}
  Suppose that $\phi$ is a handlebody group element, and let $A, B$ be the blocks of $\phi_*$ as in Lemma~\ref{lem:hanaction}. If
  \[\im(A-\Id) + B(\ker((A^t)^{-1}-\Id)) = L, \]
  then the handlebody bundle $W$ obtained as the mapping torus of $\phi$ satisfies $\hat\iota_* \colon H_1(M;\QQ) \cong H_1(W;\QQ)$ where $M=\partial W$.
\end{lemma}
\begin{proof}
  We have
  \[ H_1(M; \QQ) = \QQ \oplus \left(H_1(S_g;\QQ)\right)/(\Id-\phi_*) = \QQ \oplus (L \oplus D)/(\Id-\phi_*). \]
 By assumption, we have
 \[
 L =  {\im(A-\Id)} +  {B(\ker((A^t)^{-1}-\Id))}
 \] 
  Hence, the natural map
  \[ \QQ \oplus D/(\Id-(A^t)^{-1}) \to \QQ \oplus (L \oplus D)/(\Id-\phi_*) \]
  is surjective, and it is clearly injective.
  On the other hand, if we denote by $f_*\colon \pi_1(V)\to\pi_1(V)$ the map induced by the
  handlebody group element $\phi$, we have
  \[ H_1(W; \QQ) = \QQ \oplus \big( H_1(V_g;\QQ)/(\Id-f_*) \big) \iso \QQ \oplus D/(\Id-(A^t)^{-1}).\]
  Together, these imply the lemma.
\end{proof}
We are now ready to begin the proof of Theorem~\ref{thm:manybundles} in earnest.
\begin{proof}[Proof of Theorem~\ref{thm:manybundles}]

Let $f\colon F_g \to F_g$ be given. Up to replacing $f$ by a conjugate, we may assume
that $f$ acts on homology as
\[ f_* = \begin{pmatrix}
  \mathrm{Id} & U \\ 0 & V
\end{pmatrix} \] where $V$ does not have any eigenvalue $1$. This follows since
$\mathrm{Out}(F_n) \to \mathrm{GL}_n(\mathbb{Z})$ is surjective, and integral matrices
can be (integrally) conjugated to have this form.

Since the map $\mathcal{H}_g\to\mathrm{Out}(F_g)$ is also surjective
(e.g. \cite{Griffiths-A-surjective}) and the claims in Theorem~\ref{thm:manybundles} are invariant
under replacing $(\varphi_i)$ by $(\psi\varphi_i\psi^{-1})$, it suffices to show
the theorem under this assumption on $f_*$.

\smallskip Now, take a handlebody group element $\phi$ which acts as $f$ on $\pi_1(V)$.
Let $A$ be the matrix satisfying $(A^t)^{-1} = f_*$.
\begin{lemma}\label{lem:realise}
  With notation and assumptions as above, there is a matrix $B$ such that
  \[\im(A-\Id) + B(\ker((A^t)^{-1}-\Id)) = L, \]
and $B^t(A^t)^{-1} = A^{-1}B$.
\end{lemma}
\begin{proof}
  Under the assumptions, we have
  \[ A^t = \begin{pmatrix}
    \mathrm{Id} & Y \\ 0 & Z
  \end{pmatrix} \] where $Z$ is a $k\times k$ matrix such that $Z -
  \Id$ is injective, and $Y$ is a $(g-k)\times k$ matrix. In particular,  $Z-\Id$ is invertible over $\QQ$.

  We then have
  \[ A = \begin{pmatrix}
    \mathrm{Id} & 0 \\ Y^t & Z^t
  \end{pmatrix}, \]
  and observe that thus the image $(A-\Id)(\QQ^{2g})$ is the subspace spanned by the last $k$ basis vectors.
  Put
  \[ B = \begin{pmatrix}
    \mathrm{Id} & 0 \\ Y^t & 0
  \end{pmatrix}. \] Observe that $\ker((A^t)^{-1}-\Id) =
  \ker(A^t-\Id)$, and therefore it is the subspace spanned by the
  first $g-k$ basis vectors. Together, this implies that $\im(A-\Id)$ and
  $B(\ker((A^t)^{-1}-\Id))$ span $L$.
  In other words,  $B$ satisfies
  \[\im(A-\Id) +
  B(\ker((A^t)^{-1}-\Id)) = L, \]
as claimed.

  Furthermore, we have
  \[
  AB^t =
  \begin{pmatrix}
    \mathrm{Id} & Y \\ Y^t & Y^tY
  \end{pmatrix} = BA^t.\qedhere
  \]
\end{proof}
Now, let $B$ be a matrix as given by Lemma~\ref{lem:realise}. Since
$\mathcal{H}_g\to\mathrm{Out}(F_g)$ is surjective, we can find a handlebody group
element $\phi$ mapping to $f$. It then acts on $H_1(S;\mathbb{Q})$ as
\[
 \begin{pmatrix}
    A & B' \\ 0 & (A^t)^{-1}
 \end{pmatrix},
\]
since the lower right block describes the action on the first homology of the handlebody.
Applying Lemma~\ref{lem:kernel-realise}, we can therefore find a handlebody
group element $\varphi_0$ which acts as
\[
 \begin{pmatrix}
    A & B \\ 0 & (A^t)^{-1}
 \end{pmatrix}.
\]
By construction of $B$ and Lemma~\ref{lem:assumption-1}, the mapping
torus $W_0$ defined by $\varphi_0$ satisfies conditions i)~and~ii) of
Theorem~\ref{thm:manybundles}.  Now let $\psi$ be an element of the
kernel of the map
\[ \mathcal{H}_g \to \mathrm{Out}(F_g) \] such that $\psi\vert_{\partial V_g}$
is pseudo-Anosov and such that $\psi$ acts as the identity on
$H_1(\partial V_g;\QQ)$. Such a mapping class can for example be
constructed as the product of two Dehn twists about separating curves
bounding disks.

Observe that for all integers $n$, the mapping tori defined by the elements
$\psi^n\varphi_0$ then also satisfy i) and ii), since
they act on $H_1(S;\mathbb{Q})$ in the same way as $\varphi_0$.  On
the other hand, for large $n$, the elements
$\psi^n\varphi_0\vert_{\partial V_g}$ are pseudo-Anosov with diverging
Weil-Petersson translation length. Thus, the mapping tori defined by the boundary maps
of $\psi^n\varphi_0$ are hyperbolic manifolds, and by the main theorem of
\cite{Brock} their volumes diverge. By Mostow rigidity this implies in
particular that there are infinitely many distinct diffeomorphism
classes in the sequence.

This shows Theorem~\ref{thm:manybundles}.
\end{proof}

\bibliography{hb-polytope}

\begin{bibdiv}
\begin{biblist}

\bib{birman-matrix}{article}{
      author={Birman, Joan~S.},
       title={On the equivalence of {H}eegaard splittings of closed, orientable
  {$3$}-manifolds},
        date={1975},
       pages={137\ndash 164. Ann. of Math. Studies, No. 84},
      review={\MR{0375318}},
}

\bib{Bierietal1987}{article}{
      author={Bieri, Robert},
      author={Neumann, Walter~D.},
      author={Strebel, Ralph},
       title={A geometric invariant of discrete groups},
        date={1987},
        ISSN={0020-9910},
     journal={Invent. Math.},
      volume={90},
      number={3},
       pages={451\ndash 477},
         url={http://dx.doi.org/10.1007/BF01389175},
      review={\MR{914846 (89b:20108)}},
}

\bib{Brock}{article}{
      author={Brock, Jeffrey~F.},
       title={Weil-{P}etersson translation distance and volumes of mapping
  tori},
        date={2003},
        ISSN={1019-8385},
     journal={Comm. Anal. Geom.},
      volume={11},
      number={5},
       pages={987\ndash 999},
         url={https://doi.org/10.4310/CAG.2003.v11.n5.a6},
      review={\MR{2032506}},
}

\bib{Dowdalletal2015}{article}{
      author={Dowdall, Spencer},
      author={Kapovich, Ilya},
      author={Leininger, Christopher~J.},
       title={Dynamics on free-by-cyclic groups},
        date={2015},
        ISSN={1465-3060},
     journal={Geom. Topol.},
      volume={19},
      number={5},
       pages={2801\ndash 2899},
         url={https://doi.org/10.2140/gt.2015.19.2801},
      review={\MR{3416115}},
}

\bib{Dowdalletal2017a}{article}{
      author={Dowdall, Spencer},
      author={Kapovich, Ilya},
      author={Leininger, Christopher~J.},
       title={Endomorphisms, train track maps, and fully irreducible
  monodromies},
        date={2017},
        ISSN={1661-7207},
     journal={Groups Geom. Dyn.},
      volume={11},
      number={4},
       pages={1179\ndash 1200},
         url={https://doi.org/10.4171/GGD/425},
      review={\MR{3737279}},
}

\bib{Dowdalletal2017}{article}{
      author={Dowdall, Spencer},
      author={Kapovich, Ilya},
      author={Leininger, Christopher~J.},
       title={Mc{M}ullen polynomials and {L}ipschitz flows for free-by-cyclic
  groups},
        date={2017},
        ISSN={1435-9855},
     journal={J. Eur. Math. Soc. (JEMS)},
      volume={19},
      number={11},
       pages={3253\ndash 3353},
         url={https://doi.org/10.4171/JEMS/739},
      review={\MR{3713041}},
}

\bib{FunkeKielak2018}{article}{
      author={Funke, Florian},
      author={Kielak, Dawid},
       title={Alexander and {T}hurston norms, and the
  {B}ieri--{N}eumann--{S}trebel invariants for free-by-cyclic groups},
        date={2018},
        ISSN={1465-3060},
     journal={Geom. Topol.},
      volume={22},
      number={5},
       pages={2647\ndash 2696},
         url={https://doi.org/10.2140/gt.2018.22.2647},
      review={\MR{3811767}},
}

\bib{FriedlLueck2017}{article}{
      author={Friedl, Stefan},
      author={L{\"u}ck, Wolfgang},
       title={Universal {$L^2$}-torsion, polytopes and applications to
  3-manifolds},
        date={2017},
        ISSN={0024-6115},
     journal={Proc. Lond. Math. Soc. (3)},
      volume={114},
      number={6},
       pages={1114\ndash 1151},
         url={https://doi.org/10.1112/plms.12035},
      review={\MR{3661347}},
}

\bib{Geogheganetal2001}{article}{
      author={Geoghegan, Ross},
      author={Mihalik, Michael~L.},
      author={Sapir, Mark},
      author={Wise, Daniel~T.},
       title={Ascending {HNN} extensions of finitely generated free groups are
  {H}opfian},
        date={2001},
        ISSN={0024-6093},
     journal={Bull. London Math. Soc.},
      volume={33},
      number={3},
       pages={292\ndash 298},
         url={http://dx.doi.org/10.1017/S0024609301007986},
      review={\MR{1817768 (2002a:20029)}},
}

\bib{Griffiths-A-surjective}{article}{
      author={Griffiths, H.~B.},
       title={Automorphisms of a {$3$}-dimensional handlebody},
        date={1963/1964},
        ISSN={0025-5858},
     journal={Abh. Math. Sem. Univ. Hamburg},
      volume={26},
       pages={191\ndash 210},
         url={http://dx.doi.org/10.1007/BF02992786},
      review={\MR{0159313}},
}

\bib{survey}{article}{
      author={Hensel, Sebastian},
       title={A primer on handlebody groups},
        date={2017},
     journal={preprint},
        note={available at
  \url{http://www.mathematik.uni-muenchen.de/~hensel}},
}

\bib{HennekeKielak2018}{article}{
      author={Henneke, Fabian},
      author={Kielak, Dawid},
       title={Agrarian and {$L^2$}-invariants},
     journal={{arXiv:}1809.08470},
}

\bib{Kielak2018a}{article}{
      author={Kielak, Dawid},
       title={The {B}ieri--{N}eumann--{S}trebel invariants via {N}ewton
  polytopes},
        date={2020},
        ISSN={0020-9910},
     journal={Invent. Math.},
      volume={219},
      number={3},
       pages={1009\ndash 1068},
         url={https://doi.org/10.1007/s00222-019-00919-9},
      review={\MR{4055183}},
}

\bib{corank}{article}{
      author={Leininger, Christopher~J.},
      author={Reid, Alan~W.},
       title={The co-rank conjecture for 3-manifold groups},
        date={2002},
        ISSN={1472-2747},
     journal={Algebr. Geom. Topol.},
      volume={2},
       pages={37\ndash 50},
         url={https://doi.org/10.2140/agt.2002.2.37},
      review={\MR{1885215}},
}

\bib{Mutanguha2019}{article}{
      author={Mutanguha, Jean~Pierre},
       title={Irreducibility of a free group endomorphism is a mapping torus
  invariant},
     journal={{arXiv:}1910.04285},
}

\bib{Stallings1962}{incollection}{
      author={Stallings, John},
       title={On fibering certain {$3$}-manifolds},
        date={1962},
   booktitle={Topology of 3-manifolds and related topics ({P}roc. {T}he {U}niv.
  of {G}eorgia {I}nstitute, 1961)},
   publisher={Prentice-Hall, Englewood Cliffs, N.J.},
       pages={95\ndash 100},
      review={\MR{0158375}},
}

\bib{Thurston1986}{article}{
      author={Thurston, William~P.},
       title={A norm for the homology of {$3$}-manifolds},
        date={1986},
        ISSN={0065-9266},
     journal={Mem. Amer. Math. Soc.},
      volume={59},
      number={339},
       pages={i\ndash vi and 99\ndash 130},
      review={\MR{823443 (88h:57014)}},
}

\end{biblist}
\end{bibdiv}

\bigskip
\noindent Sebastian Hensel \hfill Dawid Kielak\\
\href{mailto:hensel@math.lmu.de?subject=Handlebody bundles and polytopes}{\texttt{hensel@math.lmu.de}} \hfill \href{mailto:dkielak@math.uni-bielefeld.de?subject=Handlebody bundles and polytopes&body=Dear Dawid,}{\texttt{dkielak@math.uni-bielefeld.de}} \\
Mathematisches Institut \hfill Fakultät für Mathematik\\
der Universität München \hfill Universität Bielefeld\\
Theresienstraße 39 \hfill Postfach 100031\\
80333 München \hfill 33501 Bielefeld \\
Germany \hfill Germany

\end{document}